\documentclass{article}

\usepackage{amsmath,amsthm,amssymb,mathtools}
\usepackage{fullpage}
\usepackage{hyperref}
\usepackage{natbib}

\newcommand{\ba}{\boldsymbol{a}}
\newcommand{\bb}{\boldsymbol{b}}

\newcommand{\bg}{\boldsymbol{g}}

\newcommand{\bp}{\boldsymbol{p}}
\newcommand{\bq}{\boldsymbol{q}}

\newcommand{\bx}{\boldsymbol{x}}
\newcommand{\bu}{\boldsymbol{u}}
\newcommand{\by}{\boldsymbol{y}}

\newcommand{\bz}{\boldsymbol{z}}

\newcommand{\bw}{\boldsymbol{w}}

\newcommand{\bv}{\boldsymbol{v}}
\newcommand{\bzero}{\boldsymbol{0}}

\newcommand{\argmin}{\mathop{\mathrm{argmin}}}

\newcommand{\field}[1]{\mathbb{#1}}

\newcommand{\R}{\field{R}}

\DeclareMathOperator{\ri}{ri}
\DeclareMathOperator{\supp}{supp}

\newtheorem{theorem}{Theorem}
\newtheorem{corollary}{Corollary}
\newtheorem{lemma}{Lemma}

\title{Last-Iterate Convergence of Optimistic Multiplicative Weight Update}
\author{Francesco Orabona\\
King Abdullah University of Science and Technology (KAUST)\\
Thuwal, 23955-6900, Kingdom of Saudi Arabia\\
\url{francesco@orabona.com}
}

\begin{document}

\maketitle

\begin{abstract}
Optimistic Gradient Descent Ascent (OGDA) and Optimistic Multiplicative-Weights Update (OMWU) are two very popular algorithms to solve convex/concave saddle-point problems, where OMWU is the non-Euclidean, entropic version of OGDA. It is known since the '80s that the last iterate of OGDA asymptotically converges to a saddle point in smooth problems. On the other hand, it is unknown if OMWU has the same property.
In this paper, I show that OMWU converges asymptotically for smooth convex-concave saddle-point problems, with a small enough constant learning rate.
The result does not require uniqueness, strict complementarity, an error bound, or initialization near a solution. The main new ingredient is a boundary argument showing that every cluster point satisfies the inactive-coordinate KKT inequalities.
The boundary argument was discovered with assistance from ChatGPT and is documented in the appendix.
\end{abstract}

\section{Introduction}

We are interested in showing the last-iterate convergence of Optimistic Multiplicative-Weights Update (OMWU) on the problem
\[
\min_{\bx \in \Delta_n} \max_{\by \in \Delta_m} \ f(\bx,\by),
\]
where $f:\Delta_n \times \Delta_m \to \R$ is convex in the first argument, concave in the second, and with Lipschitz gradient.

For the Euclidean version of OMWU, that is Optimistic Gradient Descent Ascent (OGDA),
\citet{Popov80} proved that the iterates converge to the smooth saddle-point problems when used with a constant small enough learning rate.
The convergence of the last iterate is interesting from a game-theoretic point of view because it implies that two players in a zero-sum game can converge to their optimal strategies by using optimistic online algorithms.

However, despite many partial and related results, it is unknown if the same guarantee holds for OMWU under the same generality. This gap in the knowledge is due to the fact that OGDA and OMWU have very different behaviours. From an intuitive point of view, if OGDA slows down, we know that it is because it is approaching the saddle-point. Instead, if OMWU slows down, it could just be due to the fact that the iterates are approaching the boundary of the simplex. This difference can be mathematically characterized by the fact that the distance generating function of the OMWU, the negative entropy, is not differentiable on the boundary of the simplex.

In this paper, I show a proof of convergence of the iterates of OMWU to a saddle-point of smooth problems, if used with a constant small enough learning rate. The proof starts from the classic one of \citet{Popov80}, generalized to the non-Euclidean case, and complemented with a new key lemma to show the optimality of cluster points of the trajectory. Remarkably, the key lemma in my analysis was proved with the assistance of ChatGPT 5.5, and I detail the process in Appendix~\ref{sec:llm}.

\section{Related Work}

\citet{Popov80} proposed what we call Optimistic Gradient Descent Ascent (OGDA) with two projections for saddle-point optimization, and he showed the last-iterate convergence with any fixed learning rate $\eta>0$ satisfying $\eta L < 1/3$, where $L$ is the Lipschitz constant of the gradient of the function $f$.
The version of OGDA with one projection was proposed by \citet{Malitsky15}.
The result of Popov was not widely known, so later people proved weaker results. For example, \citet{DaskalakisLSZ18} proved that the last iterate of OGDA converges to a neighborhood of the saddle-point of bilinear games. The extension to the mirror descent case of Popov's proof was done by \citet{Semenov17} to solve the more general problem of variational inequalities, but his proof has a mistake, and it is valid only for distance generating functions that are differentiable on the entire feasible set\footnote{The error was pointed out in \citet{Orabona19}.}. A similar caveat seems to be present in equation (D.7) in \citet{MertikopoulosLZFCP18}, where the differentiability of the distance generating function is needed for the implication, so it does not work for the negative entropy.\footnote{An attempt to fix their proof would just result in the optimality of the non-zero coordinates of the cluster point.}
Semenov's result (restricted to differentiable distance generating functions) was rediscovered in \citep[Theorem 4]{LeeKL21}. \citet{LeiNPW21} proves the \emph{local}, i.e., starting close enough to the saddle-point, last-iterate convergence of OMWU for saddle-point problems, under the assumption that some of the optimality conditions are satisfied in a strict way. They specifically pose the question of global last-iterate convergence as an open problem.

In the case of zero-sum two-person games, that is, bilinear games over the product of the simplexes, and assuming a unique solution, the last iterate convergence of OMWU was established by \citet{DaskalakisP18} (asymptotically) and \citet{WeiLZL21} (with a finite-time rate). However, these results hinge on the fact that zero-sum two-person games provide an error bound which, in general, is not present for generic saddle-point problems.

Recently, \citet{CaiFGCKLLZ24} showed that, even in simple matrix games, the last-iterate convergence rate of OMWU can be arbitrarily slow. This does not contradict our result, but it complements it by showing that behaviour in the finite-time regime.

To the best of my knowledge, this is the first global asymptotic last-iterate convergence proof for OMWU on smooth convex-concave saddle-point problems over products of simplexes, without uniqueness, strict complementarity, bilinearity, error-bound assumptions, or local initialization.

\section{Setting}

Let $\mathcal{X}\subseteq \R^n$ and ${\mathcal{Y}}\subseteq \R^m$ be nonempty closed convex sets and set $\mathcal{Z}:=\mathcal{X}\times {\mathcal{Y}}$. Denote by $\bz=(\bx,\by)$ any point in $\mathcal{Z}$.
Let $f:\mathcal{X}\times {\mathcal{Y}}\to \R$ be differentiable, convex in $\bx$, and concave in $\by$.
A point $\bz^\star=(\bx^\star,\by^\star)\in \mathcal{Z}$ is a saddle point if
\[
f(\bx^\star,\by)\le f(\bx^\star,\by^\star)\le f(\bx,\by^\star),
\qquad \forall (\bx,\by)\in \mathcal{X}\times {\mathcal{Y}}~.
\]
Denote the set of saddle points by $\mathcal{Z}^\star$ and assume $\mathcal{Z}^\star\neq\emptyset$.
Define
\[
G(\bz)
:=\left(\nabla_{\bx}f(\bx,\by),-\nabla_{\by}f(\bx,\by)\right)~.
\]
Let $\|\cdot\|$ be a norm on $\mathcal{Z}$ and $\|\cdot\|_\star$ its dual norm.

Throughout this paper we will assume that $G$ is $L$-Lipschitz:
\begin{equation}
\label{eq:smooth}
\|G(\bz)-G(\bw)\|_\star \le L\|\bz-\bw\|,
\qquad \forall \bz,\bw\in \mathcal{Z}~.
\end{equation}

\begin{lemma}[Saddle-point gap inequality]
\label{lemma:gap}
For every $\bz=(\bx,\by)\in \mathcal{Z}$ and every $\bz^\star=(\bx^\star,\by^\star)\in \mathcal{Z}^\star$,
\[
\langle G(\bz),\bz-\bz^\star\rangle
\ge f(\bx,\by^\star)-f(\bx^\star,\by)\ge 0~.
\]
In particular, if $\bg_t=G(\bz'_t)$, then $\langle \bg_t,\bz'_t-\bz^\star\rangle\ge 0$.
\end{lemma}
\begin{proof}
By convexity of $f(\cdot,\by)$, we have
\[
\langle \nabla_{\bx}f(\bx,\by),\bx-\bx^\star\rangle
\ge f(\bx,\by)-f(\bx^\star,\by)~.
\]
By concavity of $f(\bx,\cdot)$,
\[
\langle -\nabla_{\by}f(\bx,\by),\by-\by^\star\rangle
=\langle \nabla_{\by}f(\bx,\by),\by^\star-\by\rangle
\ge f(\bx,\by^\star)-f(\bx,\by)~.
\]
Adding the two inequalities gives the first inequality. Since $\bz^\star$ is a saddle point,
\[
f(\bx,\by^\star)\ge f(\bx^\star,\by^\star)
\ge f(\bx^\star,\by),
\]
which gives the second inequality.
\end{proof}

\paragraph{Mirror Descent}

Let $h:\mathcal{Z}\to \R$ be differentiable on a set containing the iterates of the algorithm and 1-strongly convex with respect to $\|\cdot\|$. This function will be called the \emph{distance generating function}. The Bregman divergence with respect to $h$ is
\[
B(\bu;\bz)
:=h(\bu)-h(\bz)-\langle \nabla h(\bz),\bu-\bz\rangle~.
\]
From the strong convexity of $h$, we have
\begin{equation}
\label{eq:sc}
B(\bu;\bz)
\ge \frac12\|\bu-\bz\|^2~.
\end{equation}

For $\eta>0$, we define Optimistic Mirror Descent/Ascent (OMDA) update as
\begin{align}
\bg_t &= G(\bz'_t), \label{eq:omda-g}\\
\bz_{t+1} &= \argmin_{\bz\in \mathcal{Z}}\ \eta\langle \bg_t,\bz\rangle+B(\bz;\bz_t), \label{eq:omda-z}\\
\bz'_{t+1} &= \argmin_{\bz\in \mathcal{Z}}\ \eta\langle \bg_t,\bz\rangle+B(\bz;\bz_{t+1})~. \label{eq:omda-zp}
\end{align}

The following lemma gives well-known results for the mirror update.

\begin{lemma}[Mirror descent inequalities]
\label{lemma:md}
Let
\[
\bz^+
=\argmin_{\bv\in \mathcal{Z}}\ \eta\langle \bg,\bv\rangle+B(\bv;\bz)~.
\]
Then, for all $\bu\in \mathcal{Z}$,
\begin{align}
\langle \nabla h(\bz)-\eta\bg-\nabla h(\bz^+),\bu-\bz^+\rangle &\le 0, \label{eq:md-opt}\\
B(\bu;\bz^+) &\le B(\bu;\bz)-B(\bz^+;\bz)-\eta\langle \bg,\bz^+-\bu\rangle~. \label{eq:md-three}
\end{align}
\end{lemma}
\begin{proof}
The first inequality is the first-order optimality condition for the convex minimization problem defining $\bz^+$.
For the second inequality, use the three-point identity
\[
\langle \nabla h(\bz^+)-\nabla h(\bz),\bu-\bz^+\rangle
=B(\bu;\bz)-B(\bu;\bz^+)-B(\bz^+;\bz)
\]
and combine it with \eqref{eq:md-opt}.
\end{proof}

Note that in the OMWU literature the update is sometimes written without the help of the additional iterate $\bz'_t$. These two formulations are completely equivalent as we show in Appendix~\ref{sec:equivalent}.

\section{Popov's Proof}

In this section, I review Popov's proof for the asymptotic convergence of OGDA. Note that I could have started even from more modern analysis of OGDA, but this is also a good way to advertise more the fact that \citet{Popov80} proposed the optimistic update for saddle-point optimization way before anyone else.

Let $h(\bz)=\frac12\|\bz\|_2^2$ and $\mathcal{Z}$ be closed and convex. The OMDA updates become the OGDA ones:
\begin{align}
\bg_t &=G(\bz'_t), \label{eq:ogda-g}\\
\bz_{t+1}&=\Pi_{\mathcal{Z}}(\bz_t-\eta\bg_t), \label{eq:ogda-z}\\
\bz'_{t+1}&=\Pi_{\mathcal{Z}}(\bz_{t+1}-\eta\bg_t)~. \label{eq:ogda-zp}
\end{align}

\begin{theorem}[Popov's proof]
\label{thm:popov}
Assume~\eqref{eq:smooth} with the Euclidean norm and let $0<\eta<\frac{1}{3L}$.
Then, the sequence $(\bz_t)_{t\ge 0}$ generated by \eqref{eq:ogda-g}--\eqref{eq:ogda-zp} converges to a point in $\mathcal{Z}^\star$.
\end{theorem}
\begin{proof}
Let $\bu\in \mathcal{Z}^\star$. The projection inequality gives, for all $t\ge 1$,
\begin{align*}
\|\bz_{t+1}-\bu\|^2
&\le \|\bz_t-\eta\bg_t-\bu\|^2-\|\bz_t-\eta\bg_t-\bz_{t+1}\|^2\\
&=\|\bz_t-\bu\|^2-\|\bz_{t+1}-\bz_t\|^2
-2\eta\langle \bg_t,\bz_{t+1}-\bu\rangle~.
\end{align*}
By Lemma~\ref{lemma:gap}, we have $\langle \bg_t,\bz'_t-\bu\rangle\ge 0$, so
\begin{align*}
\|\bz_{t+1}-\bu\|^2
&\le \|\bz_t-\bu\|^2-\|\bz_{t+1}-\bz_t\|^2
-2\eta\langle \bg_t,\bz_{t+1}-\bz'_t\rangle\\
&=\|\bz_t-\bu\|^2-\|\bz_{t+1}-\bz'_t\|^2-\|\bz'_t-\bz_t\|^2
-2\eta\langle \bg_t,\bz_{t+1}-\bz'_t\rangle
+2\langle \bz_t-\bz'_t,\bz_{t+1}-\bz'_t\rangle~.
\end{align*}
Since $\bz'_t=\Pi_{\mathcal{Z}}(\bz_t-\eta\bg_{t-1})$ and $\bz_{t+1}\in \mathcal{Z}$,
\[
\langle \bz_t-\eta\bg_{t-1}-\bz'_t,\bz_{t+1}-\bz'_t\rangle\le 0~.
\]
Therefore
\begin{align*}
\|\bz_{t+1}-\bu\|^2
&\le \|\bz_t-\bu\|^2-\|\bz_{t+1}-\bz'_t\|^2-\|\bz'_t-\bz_t\|^2 - 2\eta\langle \bg_t-\bg_{t-1},\bz_{t+1}-\bz'_t\rangle\\
&\le \|\bz_t-\bu\|^2-\|\bz_{t+1}-\bz'_t\|^2-\|\bz'_t-\bz_t\|^2 + 2\eta L\|\bz'_t-\bz'_{t-1}\|\,\|\bz_{t+1}-\bz'_t\|\\
&\le \|\bz_t-\bu\|^2-\|\bz_{t+1}-\bz'_t\|^2-\|\bz'_t-\bz_t\|^2 + 2\eta L(\|\bz'_t-\bz_t\|+\|\bz_t-\bz'_{t-1}\|)\|\bz_{t+1}-\bz'_t\|\\
&\le \|\bz_t-\bu\|^2+(2\eta L-1)\|\bz_{t+1}-\bz'_t\|^2+(\eta L-1)\|\bz'_t-\bz_t\|^2
+\eta L\|\bz_t-\bz'_{t-1}\|^2~.
\end{align*}
Summing from $t=N$ to $t=M$, where $1\le N\le M$, gives
\begin{align}
\|\bz_{M+1}-\bu\|^2
&+(1-\eta L)\sum_{t=N}^M\|\bz_t-\bz'_t\|^2
+(1-3\eta L)\sum_{t=N}^{M-1}\|\bz_{t+1}-\bz'_t\|^2
+(1-2\eta L)\|\bz_{M+1}-\bz'_M\|^2 \nonumber\\
&\le \|\bz_N-\bu\|^2+\eta L\|\bz_N-\bz'_{N-1}\|^2~. \label{eq:popov-tail}
\end{align}
Taking $N=1$ and using $\eta<1/(3L)$ shows that $(\bz_t)$ is bounded,
$\bz_t-\bz'_t\to \bzero$, and $\bz_{t+1}-\bz'_t\to \bzero$.
Let $\bar{\bz}$ be a cluster point of $(\bz_t)$ and choose $t_k$ such that $\bz_{t_k}\to \bar{\bz}$. Then, also $\bz'_{t_k}\to \bar{\bz}$ and $\bz_{t_k+1}\to \bar{\bz}$.
By continuity of $G$ and of the Euclidean projection, passing to the limit in
\[
\bz_{t_k+1}=\Pi_{\mathcal{Z}}(\bz_{t_k}-\eta G(\bz'_{t_k}))
\]
gives
\[
\bar{\bz}=\Pi_{\mathcal{Z}}(\bar{\bz}-\eta G(\bar{\bz}))~.
\]
This is equivalent to
\[
\langle G(\bar{\bz}),\bu-\bar{\bz}\rangle\ge 0,\qquad \forall \bu\in \mathcal{Z},
\]
and hence $\bar{\bz}\in \mathcal{Z}^\star$.

Now set $\bu=\bar{\bz}$ in \eqref{eq:popov-tail}. Since $\bz_{t_k}\to\bar{\bz}$ and $\bz_{t_k}-\bz'_{t_k-1}\to\bzero$, for every $\varepsilon>0$ we can choose $N=t_k$ large enough that
\[
\|\bz_N-\bar{\bz}\|^2+\eta L\|\bz_N-\bz'_{N-1}\|^2<\varepsilon^2~.
\]
Then, \eqref{eq:popov-tail} implies $\|\bz_{M+1}-\bar{\bz}\|<\varepsilon$ for all $M\ge N$. Hence, $\bz_t\to\bar{\bz}\in \mathcal{Z}^\star$.
\end{proof}

\section{Proof for Differentiable Distance Generating Functions}

We now show a simple generalization of Popov's proof for non-Euclidean differentiable distance generating functions. This is essentially the proof in \citet{Semenov17}.

\begin{lemma}[One-step Bregman inequality]
\label{lemma:bregman-step}
Let $(\bz_t,\bz'_t)$ be generated by \eqref{eq:omda-g}--\eqref{eq:omda-zp}. For every $\bz^\star\in \mathcal{Z}^\star$ and every $t\ge 1$,
\begin{equation}
B(\bz^\star;\bz_{t+1})
\le B(\bz^\star;\bz_t)
+(2\eta L-1)B(\bz_{t+1};\bz'_t)
+(\eta L-1)B(\bz'_t;\bz_t) +\eta L\, B(\bz_t;\bz'_{t-1})~. \label{eq:bregman-step}
\end{equation}
\end{lemma}
\begin{proof}
By \eqref{eq:md-three},
\begin{align*}
B(\bz^\star;\bz_{t+1})
&\le B(\bz^\star;\bz_t)-B(\bz_{t+1};\bz_t)
-\eta\langle \bg_t,\bz_{t+1}-\bz^\star\rangle\\
&=B(\bz^\star;\bz_t)-B(\bz_{t+1};\bz_t)
-\eta\langle \bg_t,\bz_{t+1}-\bz'_t\rangle
-\eta\langle \bg_t,\bz'_t-\bz^\star\rangle~.
\end{align*}
The three-point identity for Bregman divergences gives
\[
-B(\bz_{t+1};\bz_t)
=-B(\bz_{t+1};\bz'_t)-B(\bz'_t;\bz_t)
+\langle \nabla h(\bz_t)-\nabla h(\bz'_t),\bz_{t+1}-\bz'_t\rangle~.
\]
Moreover, applying \eqref{eq:md-opt} to the update defining $\bz'_t$ yields
\[
\langle \nabla h(\bz_t)-\eta\bg_{t-1}-\nabla h(\bz'_t),
\bz_{t+1}-\bz'_t\rangle\le 0~.
\]
Hence,
\begin{align*}
-\eta\langle \bg_t,\bz_{t+1}-\bz'_t\rangle
+\langle \nabla h(\bz_t)-\nabla h(\bz'_t),\bz_{t+1}-\bz'_t\rangle
\le -\eta\langle \bg_t-\bg_{t-1},\bz_{t+1}-\bz'_t\rangle~.
\end{align*}
Combining the previous displays,
\begin{align*}
B(\bz^\star;\bz_{t+1})
&\le B(\bz^\star;\bz_t)-B(\bz_{t+1};\bz'_t)-B(\bz'_t;\bz_t)-\eta\langle \bg_t-\bg_{t-1},\bz_{t+1}-\bz'_t\rangle-\eta\langle \bg_t,\bz'_t-\bz^\star\rangle\\
&\le B(\bz^\star;\bz_t)-B(\bz_{t+1};\bz'_t)-B(\bz'_t;\bz_t)+\eta L\|\bz'_t-\bz'_{t-1}\|\,\|\bz_{t+1}-\bz'_t\|-\eta\langle \bg_t,\bz'_t-\bz^\star\rangle\\
&\le B(\bz^\star;\bz_t)-B(\bz_{t+1};\bz'_t)-B(\bz'_t;\bz_t)+\eta L(\|\bz'_t-\bz_t\|+\|\bz_t-\bz'_{t-1}\|)\|\bz_{t+1}-\bz'_t\|\\
&\quad -\eta\langle \bg_t,\bz'_t-\bz^\star\rangle\\
&\le B(\bz^\star;\bz_t)-B(\bz_{t+1};\bz'_t)-B(\bz'_t;\bz_t)+\frac{\eta L}{2}\|\bz'_t-\bz_t\|^2+\eta L\|\bz_{t+1}-\bz'_t\|^2+\frac{\eta L}{2}\|\bz_t-\bz'_{t-1}\|^2\\
&\quad-\eta\langle \bg_t,\bz'_t-\bz^\star\rangle~.
\end{align*}
Using~\eqref{eq:sc} on the three squared norms and dropping the nonnegative term $\eta\langle \bg_t,\bz'_t-\bz^\star\rangle$, which is nonnegative by Lemma~\ref{lemma:gap}, gives~\eqref{eq:bregman-step}.
\end{proof}

\begin{lemma}[Comparator form of the one-step inequality]
\label{lemma:comparator-step}
Let $(\bz_t,\bz'_t)$ be generated by \eqref{eq:omda-g}--\eqref{eq:omda-zp}. For every $\bu\in \mathcal{Z}$ and every $t\ge 1$,
\begin{align*}
B(\bu;\bz_{t+1})
+(1-2\eta L)B(\bz_{t+1};\bz'_t)+(1-\eta L)B(\bz'_t;\bz_t)
+\eta\langle G(\bz'_t),\bz'_t-\bu\rangle
\le B(\bu;\bz_t)+\eta L\, B(\bz_t;\bz'_{t-1})~.
\end{align*}
\end{lemma}
\begin{proof}
The proof of Lemma~\ref{lemma:bregman-step} uses the saddle-point property only in its final line, where the term $\eta\langle G(\bz'_t),\bz'_t-\bz^\star\rangle$ is dropped. Repeating the same argument with an arbitrary comparator $\bu\in \mathcal{Z}$ and keeping this term gives the stated inequality.
\end{proof}

\begin{corollary}
\label{cor:bregman-consequences}
Assume $0<\eta<1/(3L)$. Let $\bz^\star\in \mathcal{Z}^\star$ and define
\[
A_t:=B(\bz_t;\bz'_{t-1}),\qquad
C_t:=B(\bz'_t;\bz_t),\qquad
D_t:=B(\bz^\star;\bz_t)
\]
for $t\ge 1$. Then, for all $M\ge 1$,
\begin{align}
D_{M+1}
&+(1-\eta L)\sum_{t=1}^M C_t
+(1-3\eta L)\sum_{t=2}^M A_t
+(1-2\eta L)A_{M+1} \le D_1+\eta L A_1~. \label{eq:bregman-summed}
\end{align}
Consequently,
\[
\sum_{t=1}^\infty B(\bz'_t;\bz_t)<\infty,
\qquad
\sum_{t=1}^\infty B(\bz_{t+1};\bz'_t)<\infty,
\]
and
\[
\bz'_t-\bz_t\to\bzero,\qquad
\bz_{t+1}-\bz'_t\to\bzero,\qquad
\bz_{t+1}-\bz_t\to\bzero~.
\]
Moreover, the sequence $(\bz_t)$ is bounded and $B(\bz^\star;\bz_t)$ converges.
\end{corollary}
\begin{proof}
Summing the inequality in Lemma~\ref{lemma:bregman-step} from $t=1$ to $M$ gives~\eqref{eq:bregman-summed}. Since $\eta<1/(3L)$, all coefficients on the left-hand side of~\eqref{eq:bregman-summed} are positive. This proves the two Bregman summability claims.
The norm convergence follows from the strong convexity lower bound~\eqref{eq:sc}.

The same one-step inequality also gives
\[
D_{t+1}+\eta L A_{t+1}
\le D_t+\eta L A_t-(1-\eta L)C_t,
\]
because $1-2\eta L\ge \eta L$. Thus, $D_t+\eta L A_t$ is nonincreasing and bounded below by zero, hence it converges. Since $A_t\to 0$, $D_t=B(\bz^\star;\bz_t)$ converges. Finally, \eqref{eq:sc} gives
\[
\frac12\|\bz_t-\bz^\star\|^2\le B(\bz^\star;\bz_t),
\]
so $(\bz_t)$ is bounded.
\end{proof}

\begin{lemma}
\label{lemma:interior-cluster}
Assume $0<\eta<1/(3L)$ and let $(\bz_t,\bz'_t)$ be generated by \eqref{eq:omda-g}--\eqref{eq:omda-zp}. Suppose that $(\bz_t)$ has a cluster point $\bar{\bz}$ at which $h$ is continuously differentiable in a neighborhood relative to $\mathcal{Z}$. Then, $\bar{\bz}\in \mathcal{Z}^\star$ and $\bz_t\to\bar{\bz}$.
\end{lemma}
\begin{proof}
Let $t_k$ be such that $\bz_{t_k}\to\bar{\bz}$. By Corollary~\ref{cor:bregman-consequences}, $\bz'_{t_k}\to\bar{\bz}$ and $\bz_{t_k+1}\to\bar{\bz}$.
Apply~\eqref{eq:md-opt} to the update defining $\bz_{t_k+1}$:
\[
\langle \nabla h(\bz_{t_k})-\eta G(\bz'_{t_k})-\nabla h(\bz_{t_k+1}),
\bu-\bz_{t_k+1}\rangle\le 0,
\qquad \forall \bu\in \mathcal{Z}~.
\]
Passing to the limit gives
\[
-\eta\langle G(\bar{\bz}),\bu-\bar{\bz}\rangle\le 0,
\qquad \forall \bu\in \mathcal{Z}~.
\]
Thus, $\bar{\bz}$ satisfies the variational inequality
\[
\langle G(\bar{\bz}),\bu-\bar{\bz}\rangle\ge 0,
\qquad \forall \bu\in \mathcal{Z},
\]
which is equivalent to $\bar{\bz}\in \mathcal{Z}^\star$ for differentiable convex-concave saddle-point problems.

By Corollary~\ref{cor:bregman-consequences}, $B(\bar{\bz};\bz_t)$ converges. Along the subsequence $t_k$, continuous differentiability at $\bar{\bz}$ gives $B(\bar{\bz};\bz_{t_k})\to 0$.
Therefore, the whole convergent scalar sequence $B(\bar{\bz};\bz_t)$ has limit zero. By \eqref{eq:sc},
\[
\frac12\|\bz_t-\bar{\bz}\|^2\le B(\bar{\bz};\bz_t)\to 0,
\]
so $\bz_t\to\bar{\bz}$.
\end{proof}

From the proof, it should be clear why this approach does not work for OMWU: the function $\psi$ for OMWU is not differentiable on the boundary of the simplex. Hence, one cannot take the required limit operation.

In the next section, we show a proof specialized to OMWU that follow a different route.

\section{Main Result}

In this section, we show our main result: the asymptotic convergence of OMWU.

\begin{theorem}
\label{thm:main}
Let $\mathcal{Z}=\Delta_n\times\Delta_m$.
Use OMWU with $0<\eta<1/(3L)$, and initialization in $\ri \mathcal{Z}$. Then, there exists $\bar{\bz}\in \mathcal{Z}^\star$ such that
\[
    \bz_t\to \bar{\bz},
    \qquad
    \bz'_t\to \bar{\bz}~.
\]
\end{theorem}

To prove this result, we first specialize our general setting to the entropic one.

\subsection{Entropy on the Two Simplexes}
\label{sec:entropy}

Let
\[
\mathcal{X}=\Delta_n:=\left\{\bx\in\R^n_{\geq 0}:\sum_{i=1}^n x_i=1\right\},
\qquad
\mathcal{Y}=\Delta_m:=\left\{\by\in\R^m_{\geq 0}:\sum_{j=1}^m y_j=1\right\}~.
\]
Throughout this section, we use the product norm
\[
\|\bz\|=\sqrt{\|\bx\|_1^2+\|\by\|_1^2}~.
\]
Let the distance generating function be defined as
\[
h_{\mathcal{X}}(\bx)=\sum_{i=1}^n x_i\ln x_i,\qquad
h_{\mathcal{Y}}(\by)=\sum_{j=1}^m y_j\ln y_j,\qquad
h(\bz)=h_{\mathcal{X}}(\bx)+h_{\mathcal{Y}}(\by)~.
\]
The associated Bregman divergence is
\[
B(\bu;\bz)
=\sum_{i=1}^n u_{\mathcal{X},i} \ln\frac{u_{\mathcal{X},i}}{x_i}
+\sum_{j=1}^m u_{\mathcal{Y},j} \ln\frac{u_{\mathcal{Y},j}}{y_j},
\]
with the convention $0\ln 0=0$. If $\bz_t,\bz'_t\in\ri \mathcal{Z}$, then the OMDA iterates remain in $\ri \mathcal{Z}$.
By Pinsker's inequality on each simplex,
\[
B(\bu;\bz)\ge
\frac12\|\bu_{\mathcal{X}}-\bx\|_1^2+\frac12\|\bu_{\mathcal{Y}}-\by\|_1^2
=\frac12\|\bu-\bz\|^2,
\]
so the negative entropy is 1-strongly convex with respect to the above product norm.
For $\bp\in\Delta_n$ and $\bq\in\Delta_m$, write
\[
B_{\mathcal{X}}(\bp;\bx)
=\sum_{i=1}^n p_i\ln\frac{p_i}{x_i},
\qquad
B_{\mathcal{Y}}(\bq;\by)
=\sum_{j=1}^m q_j\ln\frac{q_j}{y_j}~.
\]

Writing
\[
\bg_{\mathcal{X},t}=\nabla_{\bx}f(\bz'_t),\qquad
\bg_{\mathcal{Y},t}=-\nabla_{\by}f(\bz'_t),
\]
the updates are
\begin{align}
x_{t+1,i}
&=\frac{x_{t,i}\exp(-\eta g_{\mathcal{X},t,i})} {\sum_{\ell=1}^n x_{t,\ell}\exp(-\eta g_{\mathcal{X},t,\ell})}, \label{eq:x-mw}\\
y_{t+1,j}
&=\frac{y_{t,j}\exp(-\eta g_{\mathcal{Y},t,j})} {\sum_{\ell=1}^m y_{t,\ell}\exp(-\eta g_{\mathcal{Y},t,\ell})}~. \label{eq:y-mw}
\end{align}
The primed iterate $\bz'_{t+1}$ satisfies the same formulas with base point $\bz_{t+1}$ and the same vector $\bg_t=(\bg_{\mathcal{X},t}, \bg_{\mathcal{Y},t})$.

We also state the optimality condition for a minimum over the simplex, that we will use to certify a saddle point over the simplex.

\begin{lemma}[Optimality conditions over the simplex]
\label{lemma:simplex-kkt}
Let $\phi:\Delta_d\to\R$ be differentiable and convex. Let $\bar{\bp}\in\Delta_d$, $S=\supp(\bar{\bp})$, and suppose that there is a scalar $a\in\R$ such that
\[
\frac{\partial \phi(\bar{\bp})}{\partial \bar{p}_i}=a, \qquad \forall i\in S~.
\]
Then, $\bar{\bp}$ minimizes $\phi$ over $\Delta_d$ if and only if
\[
\frac{\partial \phi(\bar{\bp})}{\partial \bar{p}_i}\ge a, \qquad \forall i\notin S~.
\]
\end{lemma}
\begin{proof}
The first-order condition for minimizing the convex function $\phi$ over $\Delta_d$ is
\[
\langle \nabla\phi(\bar{\bp}),\bq-\bar{\bp}\rangle\ge 0, \qquad \forall \bq\in\Delta_d~.
\]
Using the constancy of $\frac{\partial \phi(\bar{\bp})}{\partial \bar{p}_i}$ on $S$ and $\sum_i(q_i-\bar p_i)=0$,
\[
\langle \nabla\phi(\bar{\bp}),\bq-\bar{\bp}\rangle
= \sum_{i\notin S}\left(\frac{\partial \phi(\bar{\bp})}{\partial \bar{p}_i}-a\right)q_i~.
\]
This is nonnegative for every $\bq\in\Delta_d$ if and only if $\frac{\partial \phi(\bar{\bp})}{\partial \bar{p}_i} - a \ge 0$ for every $i\notin S$.
\end{proof}

\subsection{Optimality for Active Coordinates}

We now consider the negative entropy case and assume $0<\eta<1/(3L)$. We also assume that the method is initialized in $\ri \mathcal{Z}$. Hence all finite iterates $\bz_t,\bz'_t$ remain in $\ri \mathcal{Z}$, and the multiplicative formulas \eqref{eq:x-mw}--\eqref{eq:y-mw} are valid at every finite time.


\begin{lemma}[Optimality of active coordinates of cluster points]
\label{lemma:face-stationarity}
Let $\mathcal{Z}=\Delta_n\times\Delta_m$.
Run OMWU with $0<\eta<1/(3L)$. Let $\bar{\bz}=(\bar{\bx},\bar{\by})$ be a cluster point of $(\bz_t)$. Set $S_{\bar{\bx}}:=\supp(\bar{\bx})$ and $S_{\bar{\by}}:=\supp(\bar{\by})$.
Then, there exist scalars $a_{\bar{\bx}},a_{\bar{\by}}\in\R$ such that
\[
\frac{\partial f(\bar{\bx},\bar{\by})}{\partial x_i}
=a_{\bar{\bx}},\qquad \forall i\in S_{\bar{\bx}},
\]
and
\[
-\frac{\partial f(\bar{\bx},\bar{\by})}{\partial y_j}
=a_{\bar{\by}},\qquad \forall j\in S_{\bar{\by}}~.
\]
\end{lemma}
\begin{proof}
Choose $t_k$ such that $\bz_{t_k}\to\bar{\bz}$. By Corollary~\ref{cor:bregman-consequences},
\[
\bz'_{t_k}\to\bar{\bz},\qquad \bz_{t_k+1}\to\bar{\bz}~.
\]
For the $\mathcal{X}$-coordinate update \eqref{eq:x-mw}, define
\[
\alpha_{\mathcal{X},t}
:=\ln\left(\sum_{\ell=1}^n x_{t,\ell}\exp(-\eta g_{\mathcal{X},t,\ell})\right)~.
\]
Then
\[
\ln\frac{x_{t+1,i}}{x_{t,i}}=-\eta g_{\mathcal{X},t,i}-\alpha_{\mathcal{X},t}~.
\]
If $i\in S_{\bar{\bx}}$, then $x_{t_k,i}\to\bar x_i>0$ and
$x_{t_k+1,i}\to\bar x_i>0$, so
\[
\ln\frac{x_{t_k+1,i}}{x_{t_k,i}}\to 0~.
\]
Since $\bz'_{t_k}\to\bar{\bz}$ and $G$ is continuous,
\[
\alpha_{\mathcal{X},t_k}\to -\eta \frac{\partial f(\bar{\bx},\bar{\by})}{\partial x_i},
\qquad \forall i\in S_{\bar{\bx}}~.
\]
The left-hand side does not depend on $i$, so $\frac{\partial f(\bar{\bx},\bar{\by})}{\partial x_i}$ is constant on $S_{\bar{\bx}}$. Denote this constant by $a_{\bar{\bx}}$.

The proof for the ${\mathcal{Y}}$ block is identical.
\end{proof}

At this point, the convergence proof would then be over if you could prove an inequality on the inactive coordinate of $\bar{\bz}$, as specified in the next lemma.
\begin{lemma}[Inactive signs imply the full variational inequality]
\label{lemma:inactive-signs-to-vi}
Let $\mathcal{Z}=\Delta_n\times\Delta_m$.
Run OMWU with $0<\eta<1/(3L)$.
Let $\bar{\bz}=(\bar{\bx},\bar{\by})$ be a cluster point of $(\bz_t)$, set
\[
S_{\bar{\bx}}=\supp(\bar{\bx}),\qquad S_{\bar{\by}}=\supp(\bar{\by}),
\]
and let $a_{\bar{\bx}},a_{\bar{\by}}$ be the active-face constants from
Lemma~\ref{lemma:face-stationarity}. Suppose, in addition, that
\begin{align}
\frac{\partial f(\bar{\bx},\bar{\by})}{\partial x_i}&\ge a_{\bar{\bx}},
\qquad \forall i\notin S_{\bar{\bx}}, \label{eq:inactive-x-sign}\\
-\frac{\partial f(\bar{\bx},\bar{\by})}{\partial y_j}&\ge a_{\bar{\by}},
\qquad \forall j\notin S_{\bar{\by}}~. \label{eq:inactive-y-sign}
\end{align}
Then
\[
\langle G(\bar{\bz}),\bu-\bar{\bz}\rangle\ge 0,
\qquad \forall \bu\in\Delta_n\times\Delta_m,
\]
and hence $\bar{\bz}\in \mathcal{Z}^\star$.
\end{lemma}
\begin{proof}
Let $\bu=(\bu_{\mathcal{X}},\bu_{\mathcal{Y}})\in\Delta_n\times\Delta_m$. By the active-face constancy and
$\sum_i (u_{\mathcal{X},i} - \bar x_i) = \sum_j (u_{\mathcal{Y},j} - \bar y_j)=0$,
\begin{align*}
\langle G(\bar{\bz}),\bu-\bar{\bz}\rangle
&=\sum_{i\notin S_{\bar{\bx}}}
\left(\frac{\partial f(\bar{\bx},\bar{\by})}{\partial x_i}-a_{\bar{\bx}}\right)u_{\mathcal{X},i}
+\sum_{j\notin S_{\bar{\by}}}
\left(-\frac{\partial f(\bar{\bx},\bar{\by})}{\partial y_j}-a_{\bar{\by}}\right)u_{\mathcal{Y},j}~.
\end{align*}
The right-hand side is nonnegative by \eqref{eq:inactive-x-sign} and
\eqref{eq:inactive-y-sign}. The variational inequality is equivalent to the saddle
condition for differentiable convex-concave saddle-point problems.
\end{proof}

The next section will prove this missing part.

\subsection{The Missing Piece by ChatGPT 5.5: Optimality of Inactive-Coordinates}
\label{sec:missing_piece}

For notational convenience, write
\[
    g_{{\mathcal{X}},i}(\bz):=\frac{\partial f(\bz)}{\partial x_i},
    \qquad
    g_{{\mathcal{Y}},j}(\bz):=-\frac{\partial f(\bz)}{\partial y_j}~.
\]

We shall prove that the inactive-coordinate sign conditions in Lemma~\ref{lemma:inactive-signs-to-vi} always hold. This part of the proof was proved by ChatGPT 5.5 Pro. Additional details on the failure of ChatGPT and the used prompts are in Appendix~\ref{sec:llm}.

\begin{lemma}[A residual bound near a partially verified point]
\label{lemma:residual}
Let $\bar{\bz}=(\bar{\bx},\bar{\by})\in \mathcal{Z}$, set
\[
    S_{\bar{\bx}}:=\supp(\bar{\bx}), \qquad S_{\bar{\by}}:=\supp(\bar{\by}),
\]
and suppose that there are scalars $a_{\bar{\bx}},a_{\bar{\by}}\in\R$ such that
\[
    g_{\mathcal{X},i}(\bar{\bz})=a_{\bar{\bx}}, \quad \forall i\in S_{\bar{\bx}}
    \qquad \text{ and }\qquad
    g_{\mathcal{Y},j}(\bar{\bz})=a_{\bar{\by}}, \quad \forall j\in S_{\bar{\by}}~.
\]
Define
\[
    I_{\bar{\bx}} := \{i\notin S_{\bar{\bx}}: g_{\mathcal{X},i}(\bar{\bz})<a_{\bar{\bx}}\}
    \qquad \text{ and }\qquad
    I_{\bar{\by}} := \{j\notin S_{\bar{\by}}: g_{\mathcal{Y},j}(\bar{\bz})<a_{\bar{\by}}\},
\]
and define the face
\begin{align*}
    \mathcal{F}_{\bar{\bx}}&:=\{x\in\Delta_n:x_i=0 \text{ for all } i\in I_{\bar{\bx}}\},\\
    \mathcal{F}_{\bar{\by}}&:=\{y\in\Delta_m:y_j=0 \text{ for all } j\in I_{\bar{\by}}\},\\
    \mathcal{F}_{\bar{\bz}}&:=\mathcal{F}_{\bar{\bx}}\times \mathcal{F}_{\bar{\by}}~.
\end{align*}
For $\bz=(\bx,\by)$, define
\[
    r_{\bar{\bx}}(\bx):=\sum_{i\in I_{\bar{\bx}}} x_i,
    \quad
    r_{\bar{\by}}(\by):=\sum_{j\in I_{\bar{\by}}} y_j,
    \quad \text{ and } \quad
    r(\bz):=r_{\bar{\bx}}(\bx)+r_{\bar{\by}}(\by)~.
\]
Then, there are a neighborhood $\mathcal{U}_{\bar{\bz}}$ of $\bar{\bz}$ in $\mathcal{Z}$ and a constant $C>0$ such that
\[
\langle G(\bz),\bz-\bar{\bz}\rangle \ge -C\, r(\bz), \qquad \forall \bz\in \mathcal{U}_{\bar{\bz}}~.
\]
\end{lemma}
\begin{proof}
First, we show that $\bar{\bz}$ satisfies the variational inequality on the face $\mathcal{F}_{\bar{\bz}}$. Since $I_{\bar{\bx}} \cap S_{\bar{\bx}}=\varnothing$ and $I_{\bar{\by}}\cap S_{\bar{\by}}=\varnothing$, we have $\bar{\bz}\in \mathcal{F}_{\bar{\bz}}$. For every $\bu_{\mathcal{X}} \in \mathcal{F}_{\bar{\bx}}$,
\begin{align*}
\langle \nabla_{\bx} f(\bar{\bz}),\bu_{\mathcal{X}}-\bar{\bx}\rangle
&=\sum_i g_{\mathcal{X},i}(\bar{\bz})(u_{\mathcal{X},i}-\bar x_i)
=\sum_i (g_{\mathcal{X},i}(\bar{\bz})-a_{\bar{\bx}})(u_{\mathcal{X},i}-\bar x_i)\\
&=\sum_{i\notin S_{\bar{\bx}},\,i\notin I_{\bar{\bx}}}(g_{\mathcal{X},i}(\bar{\bz})-a_{\bar{\bx}})u_{\mathcal{X},i}
\ge 0,
\end{align*}
where the second equality uses $\sum_i(u_{\mathcal{X},i}-\bar x_i)=0$, the third equality uses $\bar x_i=0$ for $i\notin S_{\bar{\bx}}$, $u_{\mathcal{X},i}=0$ for $i\in I_{\bar{\bx}}$, and $g_{\mathcal{X},i}(\bar{\bz})=a_{\bar{\bx}}$ for $i\in S_{\bar{\bx}}$, and the last inequality follows because if $i\notin S_{\bar{\bx}}$ and $i\notin I_{\bar{\bx}}$, then by the definition of $I_{\bar{\bx}}$ we have $g_{\mathcal{X},i}(\bar{\bz})\ge a_{\bar{\bx}}$.

Similarly, for every $\bu_{\mathcal{Y}} \in \mathcal{F}_{\bar{\by}}$,
\[
    \langle -\nabla_{\by} f(\bar{\bz}), \bu_{\mathcal{Y}} - \bar{\by} \rangle\ge 0~.
\]
Thus, for every $\bu\in \mathcal{F}_{\bar{\bz}}$,
\begin{equation}
    \label{eq:face_vi}
    \langle G(\bar{\bz}), \bu - \bar{\bz}\rangle\ge 0~.
\end{equation}

Next, the saddle operator $G=(\nabla_{\bx} f,-\nabla_{\by} f)$ is \emph{monotone}. Indeed,
for $\bz=(\bx,\by)$ and $\bw=(\tilde{\bx},\tilde{\by})$, convexity in $\bx$ gives
\begin{align*}
\langle \nabla_{\bx} f(\bx,\by)-\nabla_{\bx} f(\tilde{\bx},\tilde{\by}),\bx-\tilde{ \bx}\rangle
\ge f(\bx,\by)-f(\tilde{\bx},\by)+f(\tilde{\bx},\tilde{\by})-f(\bx,\tilde{\by}),
\end{align*}
and concavity in $\by$ gives
\begin{align*}
\langle -\nabla_{\by} f(\bx,\by)+\nabla_{\by} f(\tilde{\bx},\tilde{\by}),\by-\tilde{\by}\rangle
\ge f(\bx,\tilde{\by})-f(\bx,\by)+f(\tilde{\bx},\by)-f(\tilde{\bx},\tilde{\by})~.
\end{align*}
Adding these two inequalities yields
\begin{equation}
    \label{eq:monotone}
    \langle G(\bz)-G(\bw),\bz-\bw\rangle\ge 0, \qquad \forall \bz,\bw\in \mathcal{Z}~.
\end{equation}
Applying~\eqref{eq:monotone} with $\bz=\bu\in \mathcal{F}_{\bar{\bz}}$ and $\bw=\bar{\bz}$, then adding~\eqref{eq:face_vi}, gives the Minty inequality
\begin{equation}
    \label{eq:minty_face}
    \langle G(\bu),\bu-\bar{\bz}\rangle\ge 0, \qquad \forall \bu\in \mathcal{F}_{\bar{\bz}}~.
\end{equation}

Now, choose a neighborhood $\mathcal{U}_{\bar{\bz}}$ of $\bar{\bz}$ small enough that $r_{\bar{\bx}}(\bx)<1$ and $r_{\bar{\by}}(\by)<1$ for every $\bz=(\bx,\by)\in \mathcal{U}_{\bar{\bz}}$. This is possible because $r_{\bar{\bx}}(\bar{\bx})=0$ and $r_{\bar{\by}}(\bar{\by})=0$, and by continuity $r$ remains small in a sufficiently small neighborhood of $\bar{\bz}$.
For such $\bz$, define $\Pi_{\mathcal{F}_{\bar{\bz}}} \bz=(\Pi_{\mathcal{F}_{\bar{\bx}}} \bx,\Pi_{\mathcal{F}_{\bar{\by}}} \by)\in \mathcal{F}_{\bar{\bz}}$ by deleting the coordinates in $I_{\bar{\bx}}$ and renormalizing separately in the two simplexes:
\[
    (\Pi_{\mathcal{F}_{\bar{\bx}}} \bx)_i=
    \begin{cases}
    \displaystyle \frac{x_i}{1-r_{\bar{\bx}}(\bx)}, & i\notin I_{\bar{\bx}}\\[1ex]
    0, & i\in I_{\bar{\bx}}
    \end{cases},
\]
and analogously for $\by$. In one simplex, deleting mass $r_{\bar{\bx}}(\bx)$ changes the deleted coordinates by $r_{\bar{\bx}}(\bx)$ and the remaining coordinates by another $r_{\bar{\bx}}(\bx)$. Hence,
\[
    \|\bx-\Pi_{\mathcal{F}_{\bar{\bx}}} \bx\|_1=2r_{\bar{\bx}}(\bx),
    \qquad
    \|\by-\Pi_{\mathcal{F}_{\bar{\by}}} \by\|_1=2r_{\bar{\by}}(\by),
\]
and, for the product norm used in Section~\ref{sec:entropy},
\begin{equation}
    \label{eq:projection_bad_mass}
    \|\bz-\Pi_{\mathcal{F}_{\bar{\bz}}} \bz\|
    =\sqrt{\|\bx-\Pi_{\mathcal{F}_{\bar{\bx}}} \bx\|_1^2+\|\by-\Pi_{\mathcal{F}_{\bar{\by}}} \by\|_1^2}
    \leq 2r(\bz)~.
\end{equation}
Let
\[
    D:=\sup_{\bz,\bw\in \mathcal{Z}} \ \|\bz-\bw\|,
    \qquad
    M:=\sup_{\bz\in \mathcal{Z}} \ \|G(\bz)\|_\star~.
\]
Both constants are finite because $\mathcal{Z}$ is compact and $G$ is continuous. By
\eqref{eq:minty_face},
\[
    \langle G(\Pi_{\mathcal{F}_{\bar{\bz}}} \bz),\Pi_{\mathcal{F}_{\bar{\bz}}} \bz-\bar{\bz}\rangle\ge 0~.
\]
Therefore, using the Lipschitzness of $G$,
\[
\begin{aligned}
\langle G(\bz),\bz-\bar{\bz}\rangle
&=\langle G(\Pi_{\mathcal{F}_{\bar{\bz}}} \bz),\Pi_{\mathcal{F}_{\bar{\bz}}} \bz-\bar{\bz}\rangle
+\langle G(\bz)-G(\Pi_{\mathcal{F}_{\bar{\bz}}} \bz),\bz-\bar{\bz}\rangle
+\langle G(\Pi_{\mathcal{F}_{\bar{\bz}}}\bz),\bz-\Pi_{\mathcal{F}_{\bar{\bz}}}\bz\rangle\\
&\ge -L\|\bz-\Pi_{\mathcal{F}_{\bar{\bz}}}\bz\|\|\bz-\bar{\bz}\|- M \|\bz-\Pi_{\mathcal{F}_{\bar{\bz}}}\bz\|
\ge -2(LD+M)r(\bz)~.
\end{aligned}
\]
The claim follows with $C=2(LD+M)$.
\end{proof}

Next, we state a technical lemma on the KL divergence.

\begin{lemma}
\label{lemma:kl_lower}
Let $\rho>1$, and let $0\le p,q\le 1$ satisfy $q\ge \rho p$. Then,
\[
    q\ln\frac qp+(1-q)\ln\frac{1-q}{1-p}
    \ge c_\rho q,
    \qquad
    c_\rho:=\ln\rho-1+\frac1\rho>0,
\]
with the usual convention that the left-hand side is $+\infty$ when $p=0<q$.
\end{lemma}
\begin{proof}
If $q=0$, there is nothing to prove. If $p=0<q$, the left-hand side is $+\infty$. Thus, assume $0<p\le q<1$. Write $q=sp$, where $s\ge \rho$. Then
\[
    d(q,p):=q\ln\frac qp+(1-q)\ln\frac{1-q}{1-p}
    =q\ln s+(1-q)\ln\frac{1-q}{1-q/s}~.
\]
Furthermore,
\[
    \ln\frac{1-q}{1-q/s}
    =-\int_{q/s}^{q}\frac{du}{1-u}~.
\]
Since $u\le q$ on this interval,
\[
    \frac1{1-u}\le \frac1{1-q},
\]
and hence
\[
    \ln\frac{1-q}{1-q/s}
    \ge -\frac{q-q/s}{1-q}~.
\]
Multiplying by $1-q$,
\[
    (1-q)\ln\frac{1-q}{1-q/s}
    \ge -q\left(1-\frac1s\right)~.
\]
Therefore
\[
    d(q,p)
    \ge q\left(\ln s-1+\frac1s\right)~.
\]
The function $\phi(s)=\ln s-1+s^{-1}$ satisfies
\[
    \phi'(s)
    =\frac{s-1}{s^2}>0 \qquad (s>1),
\]
so $\phi(s)\ge \phi(\rho)=c_\rho$. This proves the claim for $q<1$, and the case $q=1$ follows by taking the limit $q\uparrow 1$.
\end{proof}

\begin{lemma}
\label{lemma:local_bad_mass_summable}
Let $\bar{\bz}$, $I_{\bar{\bx}}$, and $I_{\bar{\by}}$ be as in Lemma~\ref{lemma:residual}, and suppose that $I_{\bar{\bx}} \cup I_{\bar{\by}}\ne\varnothing$. There are a neighborhood $\mathcal{U}_{\bar{\bz}}$ of $\bar{\bz}$, and a constant $c_\rho>0$ such that, whenever $\bz_t,\bz'_{t-1}\in \mathcal{U}_{\bar{\bz}}$,
\begin{equation}
    \label{eq:kl_bad_mass}
    B(\bz'_t;\bz_t)\ge c_\rho r(\bz'_t)~.
\end{equation}
Consequently,
\begin{equation}
    \label{eq:local_bad_mass_summable}
    \sum_{\substack{t\ge 1:\, \bz_t\in \mathcal{U}_{\bar{\bz}},\,\bz'_{t-1}\in \mathcal{U}_{\bar{\bz}}}} r(\bz'_t)<\infty~.
\end{equation}
Moreover, whenever $\bz_t,\bz'_{t}\in \mathcal{U}_{\bar{\bz}}$, there exists $\rho>1$ such that
\begin{equation}
\label{eq:expansion}
x_{t+1,i}\geq \rho x_{t,i}, \quad \forall i \in I_{\bar{\bx}},
\quad \text{ and } \quad
y_{t+1,j}\geq \rho y_{t,j}, \quad \forall j \in I_{\bar{\by}}~.
\end{equation}
\end{lemma}
\begin{proof}
Choose $\sigma>0$ so small that
\[
    g_{\mathcal{X},i}(\bar{\bz})\le a_{\bar{\bx}}-4\sigma, \quad \forall i\in I_{\bar{\bx}},
    \qquad
    g_{\mathcal{Y},j}(\bar{\bz})\le a_{\bar{\by}}-4\sigma, \quad \forall j\in I_{\bar{\by}},
\]
where the corresponding family of inequalities is understood as vacuous if the set is empty. By continuity and finiteness of the coordinate sets, we may shrink
$\mathcal{U}_{\bar{\bz}}$ so that, for every $\bw\in \mathcal{U}_{\bar{\bz}}$,
\begin{align}
    g_{\mathcal{X},i}(\bw)&\le a_{\bar{\bx}}-3\sigma, \qquad \forall i\in I_{\bar{\bx}}, \label{eq:bad_x_gap}\\
    g_{\mathcal{X},\ell}(\bw)&\ge a_{\bar{\bx}}-\sigma, \qquad \forall \ell\notin I_{\bar{\bx}},  \label{eq:good_x_gap}\\
    g_{\mathcal{Y},j}(\bw)&\le a_{\bar{\by}}-3\sigma, \qquad \forall j\in I_{\bar{\by}},          \label{eq:bad_y_gap}\\
    g_{\mathcal{Y},\ell}(\bw)&\ge a_{\bar{\by}}-\sigma, \qquad \forall \ell\notin I_{\bar{\by}}, \label{eq:good_y_gap}
\end{align}
where the lower bounds for the coordinates outside $I_{\bar{\bx}}$ and $I_{\bar{\by}}$ follow from the fact that $g_{\mathcal{X},\ell}(\bar{\bz})\geq a_{\bar{\bx}}$ for all $\ell\notin I_{\bar{\bx}}$ and $g_{\mathcal{Y},\ell}(\bar{\bz})\geq a_{\bar{\by}}$ for all $\ell\notin I_{\bar{\by}}$.

Let
\[
    K_{\mathcal{X}}:=\sup_{\bw\in \mathcal{Z},\ell} \ e^{-\eta g_{\mathcal{X},\ell}(\bw)}<\infty,
    \qquad
    K_{\mathcal{Y}}:=\sup_{\bw\in \mathcal{Z},\ell} \ e^{-\eta g_{\mathcal{Y},\ell}(\bw)}<\infty~.
\]
Since $r_{\bar{\bx}}(\bar{\bx})=r_{\bar{\by}}(\bar{\by})=0$, shrink $\mathcal{U}_{\bar{\bz}}$ further so that, for all $\bz=(\bx,\by)\in \mathcal{U}_{\bar{\bz}}$,
\begin{align}
    r_{\bar{\bx}}(\bx)K_{\mathcal{X}}
    &\le e^{-\eta(a_{\bar{\bx}}-2\sigma)}-e^{-\eta(a_{\bar{\bx}}-\sigma)},
    \label{eq:x_den_small}\\
    r_{\bar{\by}}(\by)K_{\mathcal{Y}}
    &\le e^{-\eta(a_{\bar{\by}}-2\sigma)}-e^{-\eta(a_{\bar{\by}}-\sigma)}~.
    \label{eq:y_den_small}
\end{align}
The right-hand sides are strictly positive.

Assume $\bz_t,\bz'_{t-1}\in \mathcal{U}_{\bar{\bz}}$. The update of $x'_{t,i}$ is
\[
    x'_{t,i}
    =\frac{x_{t,i}e^{-\eta g_{\mathcal{X},i}(\bz'_{t-1})}} {\sum_{\ell=1}^n x_{t,\ell}e^{-\eta g_{\mathcal{X},\ell}(\bz'_{t-1})}}~.
\]
If $i\in I_{\bar{\bx}}$, then~\eqref{eq:bad_x_gap} gives
\[
    e^{-\eta g_{\mathcal{X},i}(\bz'_{t-1})}
    \ge e^{-\eta(a_{\bar{\bx}}-3\sigma)}~.
\]
For $\ell\notin I_{\bar{\bx}}$,~\eqref{eq:good_x_gap} gives
\[
    e^{-\eta g_{\mathcal{X},\ell}(\bz'_{t-1})}
    \le e^{-\eta(a_{\bar{\bx}}-\sigma)}~.
\]
Therefore, using~\eqref{eq:x_den_small},
\[
\sum_{\ell=1}^n x_{t,\ell}e^{-\eta g_{\mathcal{X},\ell}(\bz'_{t-1})}
\le (1-r_{\bar{\bx}}(\bx_t))e^{-\eta(a_{\bar{\bx}}-\sigma)}+r_{\bar{\bx}}(\bx_t)K_{\mathcal{X}}
\le e^{-\eta(a_{\bar{\bx}}-\sigma)}+r_{\bar{\bx}}(\bx_t)K_{\mathcal{X}}
\le e^{-\eta(a_{\bar{\bx}}-2\sigma)}~.
\]
Hence, for every $i\in I_{\bar{\bx}}$,
\[
    \frac{x'_{t,i}}{x_{t,i}}
    \ge \frac{e^{-\eta(a_{\bar{\bx}}-3\sigma)}}{e^{-\eta(a_{\bar{\bx}}-2\sigma)}}
    =e^{\eta\sigma}~.
\]
The same argument in the $y$-block gives, for every $j\in I_{\bar{\by}}$,
\[
    \frac{y'_{t,j}}{y_{t,j}}\ge e^{\eta\sigma}~.
\]

Set $\rho:=e^{\eta\sigma}>1$.
The same denominator estimate, now applied to the update from $\bz_t$ to $\bz_{t+1}$, gives \eqref{eq:expansion} whenever $\bz_t,\bz'_t\in \mathcal{U}_{\bar{\bz}}$, since that update uses the gradient evaluated at $\bz'_t$.

Summing the inequalities for the primed iterates over the bad sets gives
\begin{equation}
    \label{eq:aggregate_expansion_primed}
    r_{\bar{\bx}}(\bx'_t)\ge \rho r_{\bar{\bx}}(\bx_t),
    \qquad
    r_{\bar{\by}}(\by'_t)\ge \rho r_{\bar{\by}}(\by_t)~.
\end{equation}

By the data processing inequality for KL divergence, grouping the simplex into the two sets ``bad'' and ``not bad'' cannot increase KL. Thus,
\[
    B_{\mathcal{X}}(\bx'_t;\bx_t)
    \ge d(r_{\bar{\bx}}(\bx'_t),r_{\bar{\bx}}(\bx_t)),
\]
where $d(q,p)=q\ln(q/p)+(1-q)\ln((1-q)/(1-p))$, and similarly in the $y$-block. Hence, Lemma~\ref{lemma:kl_lower} and~\eqref{eq:aggregate_expansion_primed} imply
\[
    B_{\mathcal{X}}(\bx'_t;\bx_t)\ge c_\rho r_{\bar{\bx}}(\bx'_t),
    \qquad
    B_{\mathcal{Y}}(\by'_t;\by_t)\ge c_\rho r_{\bar{\by}}(\by'_t)~.
\]
Adding the two bounds proves \eqref{eq:kl_bad_mass}. Finally, Corollary~\ref{cor:bregman-consequences} gives $\sum_{t=1}^\infty B(\bz'_t;\bz_t)<\infty$, and therefore~\eqref{eq:local_bad_mass_summable} follows from~\eqref{eq:kl_bad_mass}.
\end{proof}

\begin{lemma}[Inactive-coordinate signs for cluster points]
\label{lemma:inactive}
Let $\mathcal{Z}=\Delta_n\times\Delta_m$.
Run OMWU with $0<\eta<1/(3L)$.
Let $\bar{\bz}=(\bar{\bx},\bar{\by})$ be any cluster point of $(\bz_t)$, set
\[
    S_{\bar{\bx}} := \supp(\bar{\bx}), \qquad S_{\bar{\by}}:=\supp(\bar{\by}),
\]
and let $a_{\bar{\bx}},a_{\bar{\by}}$ be the active-face constants from Lemma~\ref{lemma:face-stationarity}. Then
\[
    \nabla_{x_i}f(\bar{\bx},\bar{\by})
    \ge a_{\bar{\bx}}, \quad \forall i\notin S_{\bar{\bx}},
    \qquad \text{ and }
    \qquad
    -\nabla_{y_j}f(\bar{\bx},\bar{\by})
    \ge a_{\bar{\by}}, \quad \forall j\notin S_{\bar{\by}}~.
\]
\end{lemma}
\begin{proof}
Suppose, toward a contradiction, that at least one inactive sign fails. Define
\[
    I_{\bar{\bx}}:=\{i\notin S_{\bar{\bx}}:g_{\mathcal{X},i}(\bar{\bz})<a_{\bar{\bx}}\},
    \qquad
    I_{\bar{\by}}:=\{j\notin S_{\bar{\by}}:g_{\mathcal{Y},j}(\bar{\bz})<a_{\bar{\by}}\}~.
\]
Then, $I_{\bar{\bx}} \cup I_{\bar{\by}}\ne\varnothing$.
Intersecting the neighborhoods given by Lemmas~\ref{lemma:residual} and \ref{lemma:local_bad_mass_summable}, and shrinking if necessary, give a neighborhood $\mathcal{U}_{\bar{\bz}}$ of $\bar{\bz}$, constants $C>0$, and $c_\rho>0$, such that
\begin{equation}
    \label{eq:local_residual_final}
    \langle G(\bz),\bz-\bar{\bz}\rangle
    \ge -Cr(\bz), \qquad \forall \bz\in \mathcal{U}_{\bar{\bz}},
\end{equation}
and
\begin{equation}
    \label{eq:local_tail_final}
    \sum_{\substack{t\ge 1:\,\bz_t\in \mathcal{U}_{\bar{\bz}},\,\bz'_{t-1}\in \mathcal{U}_{\bar{\bz}}}} r(\bz'_t)<\infty~.
\end{equation}
We also shrink $\mathcal{U}_{\bar{\bz}}$, if necessary, so that
\begin{equation}
    \label{eq:bad_mass_small_in_U}
    r_{\bar{\bx}}(\bx)<\frac12,
    \qquad
    r_{\bar{\by}}(\by)<\frac12,
    \qquad \forall \bz=(\bx,\by)\in \mathcal{U}_{\bar{\bz}}~.
\end{equation}
This is possible because $r_{\bar{\bx}}(\bar{\bx})=r_{\bar{\by}}(\bar{\by})=0$.

We now prove that a sufficiently late visit near $\bar{\bz}$ traps the whole future trajectory near $\bar{\bz}$. Apply Lemma~\ref{lemma:comparator-step} with the fixed comparator $\bu=\bar{\bz}$. Since $0<\eta<1/(3L)$, the Bregman coefficients in Lemma~\ref{lemma:comparator-step} are nonnegative, and dropping them gives
\[
    B(\bar{\bz};\bz_{t+1})
    \le B(\bar{\bz};\bz_t)-\eta\langle G(\bz'_t),\bz'_t-\bar{\bz}\rangle +\eta L\, B(\bz_t;\bz'_{t-1})~.
\]
Whenever $\bz'_t\in \mathcal{U}_{\bar{\bz}}$, \eqref{eq:local_residual_final} gives
\begin{equation}
    \label{eq:trap_step}
    B(\bar{\bz};\bz_{t+1})
    \le B(\bar{\bz};\bz_t) +\eta C\, r(\bz'_t) +\eta L\, B(\bz_t;\bz'_{t-1})~.
\end{equation}

Choose $R>0$ so that the closed product-norm ball $\overline{\mathcal B_R(\bar{\bz})}\cap \mathcal{Z}$ is contained in $\mathcal{U}_{\bar{\bz}}$, and set
\[
    \mathcal{W}:=\mathcal B_{R/2}(\bar{\bz})\cap \mathcal{Z}~.
\]
Corollary~\ref{cor:bregman-consequences} gives $\bz'_t-\bz_t\to 0$ and $\bz_t-\bz'_{t-1}\to 0$.
Hence, there is $T$ such that, for all $t\ge T$,
\begin{equation}
    \label{eq:adjacent_small}
    \|\bz'_t-\bz_t\|<\frac R4,
    \qquad
    \|\bz_t-\bz'_{t-1}\|<\frac R4~.
\end{equation}
Thus, if $t\ge T$ and $\bz_t\in \mathcal{W}$, then $\bz'_t\in \mathcal{U}_{\bar{\bz}}$ and $\bz'_{t-1}\in \mathcal{U}_{\bar{\bz}}$.

Since $\bar{\bz}$ is a cluster point, choose $t_k\to\infty$ such that $\bz_{t_k}\to\bar{\bz}$. Along this subsequence, $B(\bar{\bz};\bz_{t_k})\to 0$, because the KL divergence $B(\bar{\bz};\bz)$ only contains coordinates in $\supp(\bar{\bx})$ and $\supp(\bar{\by})$, and these coordinates converge to their positive limits along the subsequence. By~\eqref{eq:local_tail_final},
\[
    \sum_{\substack{t\ge N:\,\bz_t\in \mathcal{U}_{\bar{\bz}},\,\bz'_{t-1}\in \mathcal{U}_{\bar{\bz}}}} r(\bz'_t)\to 0
    \quad \text{ when } \quad N\to\infty,
\]
and by Corollary~\ref{cor:bregman-consequences},
\[
    \sum_{t=N}^\infty B(\bz_t;\bz'_{t-1})\to 0, \quad \text{ when } \quad N\to\infty~.
\]
Choose $N=t_k\ge T$ so large that
\begin{equation}
    \label{eq:small_tail_choice}
    B(\bar{\bz};\bz_N)
    +\eta C\sum_{\substack{t\ge N:\,\bz_t\in \mathcal{U}_{\bar{\bz}},\,\bz'_{t-1}\in \mathcal{U}_{\bar{\bz}}}} r(\bz'_t)
    +\eta L\sum_{t=N}^\infty B(\bz_t;\bz'_{t-1})
    <\frac{R^2}{8}~.
\end{equation}
By Pinsker's inequality,
\[
    \frac12\|\bz_N-\bar{\bz}\|^2\le B(\bar{\bz};\bz_N)
    <\frac{R^2}{8},
\]
so $\bz_N\in \mathcal{W}$.

We claim that $\bz_t\in \mathcal{W}$ for every $t\ge N$. If not, let $M>N$ be the first index such that $\bz_M\notin \mathcal{W}$. Then $\bz_t\in \mathcal{W}$ for $t=N,\dots,M-1$. By~\eqref{eq:adjacent_small}, for these times $\bz'_t\in \mathcal{U}_{\bar{\bz}}$ and $\bz'_{t-1}\in \mathcal{U}_{\bar{\bz}}$. Hence \eqref{eq:trap_step} applies, and each such time is included in the local sum in \eqref{eq:small_tail_choice}.
Summing \eqref{eq:trap_step} from $t=N$ to $M-1$ gives
\begin{align*}
    B(\bar{\bz};\bz_M)
    \le B(\bar{\bz};\bz_N)
    +\eta C\sum_{t=N}^{M-1}r(\bz'_t)
    +\eta L\sum_{t=N}^{M-1}B(\bz_t;\bz'_{t-1})
    <\frac{R^2}{8}~.
\end{align*}
Another application of Pinsker gives
\[
    \|\bz_M-\bar{\bz}\|<\frac R2,
\]
so $\bz_M\in \mathcal{W}$, contradicting the definition of $M$. Therefore,
\begin{equation}
    \label{eq:permanent_trap}
    \bz_t\in \mathcal{W}, \quad \forall t\ge N~.
\end{equation}
By \eqref{eq:adjacent_small}, after increasing $N$ if necessary,
\begin{equation}
    \label{eq:primed_in_U_forever}
    \bz'_t\in \mathcal{U}_{\bar{\bz}}, \quad \forall t\ge N~.
\end{equation}

It remains to contradict the local multiplicative expansion. For all $t \geq N $, by~\eqref{eq:permanent_trap} and~\eqref{eq:primed_in_U_forever}, Lemma~\ref{lemma:local_bad_mass_summable} also gives
\begin{equation}
    \label{eq:unprimed_expansion}
    x_{t+1,i}\ge \rho x_{t,i}, \quad \forall i\in I_{\bar{\bx}},
    \qquad
    y_{t+1,j}\ge \rho y_{t,j}, \quad \forall j\in I_{\bar{\by}},
\end{equation}
where $\rho>1$.

If $I_{\bar{\bx}} \ne\varnothing$, then the full-support initialization implies
\[
    r_{\bar{\bx}}(\bx_{N})=\sum_{i\in I_{\bar{\bx}}} x_{N,i}>0~.
\]
Summing \eqref{eq:unprimed_expansion} over $i\in I_{\bar{\bx}}$ gives
\[
    r_{\bar{\bx}}(\bx_{t+1}) \ge \rho r_{\bar{\bx}}(\bx_{t}), \quad \forall t\ge N,
\]
and hence
\[
    r_{\bar{\bx}}(\bx_{t}) \ge \rho^{t-N} r_{\bar{\bx}}(\bx_{N}), \quad \forall t\ge N~.
\]
Since $\rho>1$, the right-hand side eventually exceeds $1/2$, contradicting~\eqref{eq:bad_mass_small_in_U} and~\eqref{eq:permanent_trap}. Hence
$I_{\bar{\bx}}=\varnothing$.

The same argument applies to $I_{\bar{\by}}$, therefore $I_{\bar{\by}}=\varnothing$.
\end{proof}

The above lemma proves that no strictly bad inactive coordinate exists, which is exactly the desired pair of inactive-coordinate inequalities.

We can now prove the main result.
\begin{proof}[Proof of Theorem~\ref{thm:main}]
The product simplex $\mathcal{Z}$ is compact, so $(\bz_t)$ has a cluster point $\bar{\bz}$. Lemma~\ref{lemma:inactive} proves that $\bar{\bz}$ satisfies the inactive-coordinate sign conditions of Lemma~\ref{lemma:inactive-signs-to-vi}. Hence $\bar{\bz}\in \mathcal{Z}^\star$.

Choose $t_k\to\infty$ such that $\bz_{t_k}\to\bar{\bz}$. Since $\bar{\bz}\in \mathcal{Z}^\star$, Corollary~\ref{cor:bregman-consequences} implies that the scalar sequence $B(\bar{\bz};\bz_t)$ converges. Along the subsequence,
\[
    B(\bar{\bz};\bz_{t_k})
    =\sum_{i:\bar x_i>0}\bar x_i\ln\frac{\bar x_i}{x_{t_k,i}}
    +\sum_{j:\bar y_j>0}\bar y_j\ln\frac{\bar y_j}{y_{t_k,j}}
    \to 0,
\]
because the coordinates in $\supp(\bar{\bx})$ and $\supp(\bar{\by})$ converge to
positive limits, while the coordinates outside these supports have zero
coefficient in the KL divergence. Therefore the whole convergent scalar
sequence satisfies $B(\bar{\bz};\bz_t)\to 0$.
Pinsker's inequality gives
\[
    \frac12\|\bz_t-\bar{\bz}\|^2\le B(\bar{\bz};\bz_t)\to 0,
\]
so $\bz_t\to\bar{\bz}$. Finally, Corollary~\ref{cor:bregman-consequences} gives $\bz'_t-\bz_t\to0$, and hence $\bz'_t\to\bar{\bz}$.
\end{proof}

\section*{Acknowledgements}
I thank Aryan Mokhtari for telling me about Popov's algorithm in 2022.

\bibliographystyle{plainnat_nourl}
\bibliography{../../learning}

\appendix

\section{Equivalence between One- and Two-sequence OMWU}
\label{sec:equivalent}

Define
\[
T(\bq,\ba):=\argmin_{\bz \in \mathcal{Z}} \ \eta \langle \ba,\bz\rangle + B(\bz;\bq)~.
\]
Then, the OMDA scheme is
\begin{align*}
\bg_t&=G(\bz'_t),\\
\bz_{t+1} &= T(\bz_t,\bg_t),\\
\bz'_{t+1} &= T(\bz_{t+1},\bg_t)~.
\end{align*}

Note that for the entropic case, we have $T(T(\bq,\ba),\bb)=T(\bq,\ba+\bb)$.
Now, let $\bq_0,\bq_1\in\ri \mathcal{Z}$ be the initial one-sequence OMWU points, and set $\bg_0=G(\bq_0)$. Define an auxiliary two-sequence initialization by
\[
    \bz'_0:=\bq_0,
    \qquad
    \bz_0:=T(\bq_1,-2\bg_0)~.
\]
Since $\bq_1\in\ri \mathcal{Z}$, also $\bz_0\in\ri\mathcal{Z}$.
Runnig the two-sequence OMDA scheme, we have
\[
    \bz_1=T(\bz_0,\bg_0)=T(\bq_1,-\bg_0),
\]
and hence
\[
    \bz'_1=T(\bz_1,\bg_0)=\bq_1~.
\]
We prove by induction that
\[
    \bz'_t=\bq_t,\quad \forall t\ge0~.
\]
The claim is true for $t=0,1$. Suppose it is true up to time $t\ge1$.
Then, $\bg_t=G(\bz'_t)=G(\bq_t)$, and
\[
    \bz_t=T(\bq_t,-\bg_{t-1})~.
\]
Therefore,
\begin{align*}
    \bz'_{t+1}
    &=T(\bz_{t+1},\bg_t)
    =T(T(\bz_t,\bg_t),\bg_t)
    =T(\bz_t,2\bg_t)
    =T(T(\bq_t,-\bg_{t-1}),2\bg_t)
    =T(\bq_t,2\bg_t-\bg_{t-1})
    =\bq_{t+1}~.
\end{align*}
That is, $\bq_{t+1}=T(\bq_t,2G(\bq_t)-G(\bq_{t-1}))$, the one-sequence OMWU.
Thus, the one-sequence OMWU trajectory is exactly the primed trajectory of a valid two-sequence OMDA run. By Theorem 2, the two-sequence run satisfies $\bz'_t\to \bar{\bz}\in Z^\star$. Since $\bq_t=\bz'_t$, we conclude $\bq_t\to \bar{\bz}\in Z^\star$.

Note that different papers initialize OMWU differently, for example by specifying $\bq_{-1},\bq_0$, or by taking $G(\bq_{-1})=\boldsymbol{0}$, or by doing one ordinary multiplicative-weights step first. This does not matter. Once the first two full-support iterates $\bq_{0},\bq_1$ are fixed, one can analyze the convergence from t=1 onward, and changing finitely many initial steps does not affect convergence.

\section{LLM-Assisted Proof Discovery}
\label{sec:llm}

I have tried to solve this problem on and off from 2022. I tried many different approaches and all of them hit the same wall: establishing the optimality of the non-active coordinates of the cluster point. Hence, recently I decided to use ChatGPT.

The ``Thinking'' version, among many different failed attempts, suggested the same approach in Section~\ref{sec:missing_piece}, but it failed to formalized it and it gave up. Surprisingly, the Pro version also \emph{failed}, no matter what prompt I used and which strategy I suggested. The Pro version did suggest the same approach of Section~\ref{sec:missing_piece}, but again it failed to formalize it.

So, I used a different approach: I wrote a PDF detailing everything I knew about this problem, and fed it to ChatGPT. The prompts were simple enough:\footnote{These prompts were inspired by the ones in this tweet by Sebastien Bubeck: \url{https://x.com/SebastienBubeck/status/2062335588720889866}.}
\begin{itemize}
\item Study the following pdf (roughly, just get the general idea).
\item Prove the remaining part to finish the proof. Any step you make must be properly justified, do not give vague statements.
\item Great. Now take another pass through your writeup and check that things work. If something is a bit fuzzy, explain it a bit more.
\end{itemize}
In this way, ChatGPT Pro was able to produce a correct proof. The critical ingredient seems to be Lemma~\ref{lemma:comparator-step}, that was in the PDF I wrote. This is an obvious result, and I checked that the Thinking version can easily derive it with a minimal-information prompt (i.e., ``In the entropic proof, what happens if you choose z* not to be a saddle-point?''). However, it is not typically used because there is the last term is not easily controllable. I added it because I was exploring different, unbeaten paths. Indeed, this lemma seems to be necessary to prove the trapping argument, that is the pillar of the proof.

I speculate that the above success/failures might suggest different prompting strategies or harnesses for theorem proving. For example, one might first encourages the model to produce a number of \emph{possibly} useful results, and, later, trying to solve the initial task.

\end{document}